\documentclass[reqno]{amsart}

\usepackage{enumerate, hyperref}

\numberwithin{equation}{section}

\newtheorem{thm}{Theorem}[section]
\newtheorem{lem}[thm]{Lemma}

\newtheorem{cor}[thm]{Corollary}

\theoremstyle{definition}
\newtheorem{rem}[thm]{Remark}

\newcommand\R{{\mathbb R}}
\newcommand\C{{\mathbb C}}

\newcommand\goto{\mathop{\longrightarrow}}
\newcommand\Tma{T_{\mathrm{max}}}
\newcommand\union{\mathop{\cup}}

\begin{document}

\title{Continuous dependence for NLS in fractional order spaces}

\author[Thierry Cazenave]{Thierry Cazenave$^1$}

\author[Daoyuan Fang]{Daoyuan Fang$^{2}$}

\author[Zheng Han]{Zheng Han $^{2,3}$}

\address{$^{1,2}$
Department of Mathematics, Zhejiang University, Hangzhou, Zhejiang 310027, People's Republic of China}

\curraddr[Thierry Cazenave]{{\sc Universit\'e Pierre et Marie Curie \& CNRS, Laboratoire Jacques-Louis Lions,
B.C.~187, 4 place Jussieu, 75252 Paris Cedex 05, France}}

\email[Thierry Cazenave]{\href{mailto:thierry.cazenave@upmc.fr}{thierry.cazenave@upmc.fr}}
\email[Daoyuan Fang]{\href{mailto:dyf@zju.edu.cn}{dyf@zju.edu.cn}}
\email[Han Zheng]{\href{mailto:hanzheng5400@yahoo.com.cn}{hanzheng5400@yahoo.com.cn}}

\thanks{$^1$Research supported in part by Zhejiang University's Pao Yu-Kong
International Fund}
\thanks{$^2$Research supported by NSFC 10871175, 10931007, and Zhejiang NSFC Z6100217}

\thanks{$^3$Corresponding author}



\begin{abstract}
For the nonlinear Schr\"o\-din\-ger equation
 $iu_t+ \Delta u+ \lambda  |u|^\alpha u=0$ in $\R^N $,
 local existence of solutions in $H^s$ is well known
  in the $H^s$-subcritical and critical cases $0<\alpha \le 4/(N-2s)$, where $0<s<
  \min\{ N/2, 1\}$.
However, even though the solution is constructed by a fixed-point technique,
continuous dependence in $H^s$ does not follow from the contraction mapping argument.
In this paper, we show that the solution depends continuously on the initial value in the sense that the local flow is continuous $H^s \to H^s$. If, in addition, $\alpha \ge 1$ then the flow is locally Lipschitz.
\end{abstract}

\maketitle

\section{Introduction} \label{Intro}
In this paper, we study the continuity of the solution map $\varphi \mapsto u$ for
 the nonlinear Schr\"o\-din\-ger equation
\begin{equation} \label{NLS} \tag{NLS}
\begin{cases}
iu_t +\Delta u+ g(u)=0, \\ u(0)= \varphi ,
\end{cases}
\end{equation}
in $H^s (\R^N ) $, where $N\ge 1$ and
\begin{equation} \label{fHypu}
0 < s < \min \Bigl\{ 1, \frac {N} {2} \Bigr\}.
\end{equation}
We assume that the nonlinearity $g$ satisfies
\begin{equation} \label{fHypd}
g\in C^1(\C, \C),\quad g(0)= 0,
\end{equation}
and
\begin{equation} \label{fHypt}
 |g'(u) |\le A+ B  |u|^\alpha ,
\end{equation}
for all $u\in \C$, where
\begin{equation} \label{fHypq}
0<\alpha \le  \frac {4} {N-2s}.
\end{equation}
Under these assumptions, for every
initial value $\varphi \in H^s (\R^N ) $, there exists a local solution $u\in C([0,T], H^s (\R^N ) )$ of~\eqref{NLS}, which is unique under an appropriate auxiliary condition.
More precisely,  the following result is well known. (Here and in the rest of the paper,
 a pair $(q,r)$ is called admissible if $2\le r< 2N/(N-2)$ ($2\le r<\infty $ if $N=1,2$) and $2/q= N(1/2 -1/r)$.)

\begin{thm}[\cite{CWHs,Katou,CLN}] \label{echysecpzu}
Assume~\eqref{fHypu}--\eqref{fHypq} and let
 $(\gamma ,\rho )$ be the admissible pair
 defined by
\begin{equation} \label{fchysecpzt}
\rho = \frac {N(\alpha +2)} {N+s\alpha },\quad \gamma =
\frac {4(\alpha +2)} {\alpha (N-2s)}.
\end{equation}
Given $\varphi \in H^s(\R^N) $, there exist $0<\Tma \le \infty $
and a unique, maximal solution $u\in
C([0,\Tma), H^s(\R^N ))\cap L^\gamma _{\rm loc}([0,\Tma), B^s _{\rho
,2} (\R^N ))$ of~\eqref{NLS}, in the sense that
\begin{equation} \label{INLS}
u(t)= e^{it\Delta }\varphi  +i \int _0^t e^{i(t-s) \Delta } g(u(s))\,ds,
\end{equation}
for all $0\le t<\Tma$, where
$(e^{it\Delta }) _{ t\in \R }$ is the Schr\"o\-din\-ger group.
Moreover, $u\in L^q _{\rm loc}([0, \Tma),B^s_{ r,2} (\R^N ))$ for
every admissible pair $(q,r)$ and $u$ depends continuously on $\varphi $ in the following
sense. There exists a time $T \in (0, \Tma)$ such that
 if $\varphi_n\to \varphi$ in $H^s(\R^N )$ and if
$u_n$ denotes the solution of~\eqref{NLS} with the initial value
$\varphi_n$, then $ \Tma(\varphi _n) >T$ for all
sufficiently large $n$ and $u_n$ is bounded in $L^q ((0,
T),B^s_{ r,2} (\R^N ))$ for any admissible pair $(q,r)$. Moreover, $u_n\to u$
in $L^q ((0,T), B^{s-\varepsilon }_{ r,2} (\R^N ))$ as $n\to \infty $ for all $\varepsilon >0$ and all admissible pairs $(q,r)$.
 In particular, $u_n\to u$
in $C([0,T], H^{s-\varepsilon }(\R^N ))$ for all $\varepsilon >0$.
\end{thm}

Theorem~\ref{echysecpzu} goes back to~\cite{CWHs,Katou}. The precise statement which we give here is taken from\cite[Theorems~4.9.1 and~4.9.7]{CLN}.
The admissible pair $(\gamma ,\rho )$ corresponds to one particular choice of an auxiliary space, which ensures that the equation~\eqref{NLS} makes sense and that
 the solution is unique. See~\cite{CWHs,Katou,CLN} for details.
Under certain conditions on $\alpha $ and $s$, an auxiliary space is not necessary:
the equation makes sense by Sobolev's embedding and
uniqueness in $C([0,T], H^s (\R^N ) )$ holds ``unconditionally", see~\cite{Katou, FurioliT, Rogers, WinT}.

We note that the continuous dependence statement in  Theorem~\ref{echysecpzu} is weaker
than the expected one (i.e. with $\varepsilon =0$).
Indeed, Theorem~\ref{echysecpzu} is proved
by applying  a fixed point argument to equation~\eqref{INLS}, so one would expect that the
dependence of the solution  on the initial value is locally Lipschitz.
However, the metric space in which one applies
Banach's fixed point theorem involves Sobolev (or Besov) norms of order $s$, while the distance only involves Lebesgue norms. (The reason for that choice of the distance is that the nonlinearity need not be locally Lipschitz for Sobolev or Besov norms of positive order.)
Thus the flow is locally Lipschitz for Lebesgue norms, and continuous dependence
in the sense of Theorem~\ref{echysecpzu} follows by interpolation inequalities. See~\cite{CWHs,Katou} for details.

In this paper, we show that continuous dependence holds in $H^s$ in the standard sense
under the assumptions of Theorem~\ref{echysecpzu} (with an extra condition in the critical case
$\alpha = \frac {4} {N-2s}$). More precisely, our main result  is the following.

\begin{thm} \label{eMain}
Assume~\eqref{fHypu}--\eqref{fHypq} and suppose further that
\begin{equation} \label{fSbCr}
A=0  \quad  \text{if}\quad  \alpha = \frac {4} {N-2s},
\end{equation}
where $A$ is the constant in~\eqref{fHypt}.
The  solution of~\eqref{NLS} given by Theorem~$\ref{echysecpzu}$
depends continuously on $\varphi $ in the following sense.
\begin{enumerate}  [{\rm (i)}]

\item \label{eMain:u}
The mapping $\varphi \mapsto \Tma (\varphi )$
is lower semicontinuous $H^s (\R^N ) \to (0,\infty ]$.

\item \label{eMain:d}
If $\varphi_n\to \varphi$ in $H^s(\R^N )$ and if
$u_n$ (respectively, $u$) denotes the solution of~\eqref{NLS} with the initial value
$\varphi_n$ (respectively, $\varphi $), then  $u_n\to u$
in $L^q ((0, T),B^s_{ r,2} (\R^N ))$ for all
 $0<T< \Tma(\varphi )$ and all admissible pair $(q,r)$.
 In particular, $u_n\to u$
in $C([0,T], H^{s }(\R^N ))$.
\end{enumerate}
\end{thm}

Theorem~\ref{eMain} applies in particular to the model case $g(u)= \lambda  |u|^\alpha u$ with $\lambda \in \C$ and $0<\alpha \le 4/(N-2s)$.
If, in addition, $\alpha \ge 1$ then the dependence is in fact locally Lipschitz.
We summarize the corresponding results in the following corollary.

\begin{cor} \label{eMaind}
Assume~\eqref{fHypu} and
let $g(u)= \lambda  |u|^\alpha u$ with $\lambda \in \C$ and $0<\alpha \le \frac {4} {N-2s}$.
It follows that the  solution of~\eqref{NLS} given by Theorem~$\ref{echysecpzu}$
depends continuously on the initial value in the sense of Theorem~$\ref{eMain}$.
If, in addition, $\alpha \ge 1$ then the dependence is locally Lipschitz.
More precisely, let $\varphi \in H^s (\R^N ) $ and let $u$ be the corresponding solution of~\eqref{NLS}. Given $0<T<\Tma (\varphi )$ there exists $\delta >0$ such that if $\psi \in H^s (\R^N ) $ satisfies $ \|\varphi -\psi \| _{ H^s }\le \delta $ and $v$ is the corresponding solution of~\eqref{NLS}, then for every admissible pair $(q,r)$
\begin{equation}  \label{eMaind:u}
 \|u -v\|  _{ L^q ((0,T), B^s _{ r,2 }) } \le C  \|\varphi -\psi \| _{ H^s },
\end{equation}
where $C$ depends on $\varphi ,T, q,r$.
In particular, $  \|u -v\|  _{ L^\infty  ((0,T), H^s) } \le C  \|\varphi -\psi \| _{ H^s }$.
\end{cor}

We are not aware of any previous continuous dependence result for~\eqref{NLS}  in $H^s$ with noninteger $s$.
For integer $s$, the known results in the model case $g(u)=  \lambda  |u|^\alpha u$, $\lambda \in \C$ are the following. Continuous dependence in $L^2 (\R^N ) $ follows from~\cite{Tsutsumi} in the subcritical case $N\alpha <4$ and from~\cite{CWHs} in the critical case $N\alpha =4$.  Continuous dependence in $H^1 (\R^N ) $ is proved in~\cite{Katod} in the subcritical case $
(N-2)\alpha <4$ and in~\cite{CWHs, KenigM, TaoV, KillipV} in the critical case  $
(N-2)\alpha =4$. Continuous dependence in $H^2 (\R^N ) $ follows from~\cite{Katod} in the subcritical case $(N-4)\alpha <4$.

Our proof of Theorem~\ref{eMain} is based on the method used by Kato~\cite{Katod} to prove continuous dependence in $H^1 (\R^N ) $.
We briefly recall Kato's argument. Convergence in Lebesgue spaces holds by the
contraction mapping estimates, so the tricky part is the gradient estimate. By applying Strichartz estimates, this amounts in controlling $\nabla [g(u_n)- g(u)]$
in some $L^p$ space. However, $\nabla [g(u_n)- g(u)] = g'(u_n) [\nabla u_n -\nabla u]+ [g'(u_n) -g'(u)] \nabla u$. The term $g'(u_n) [\nabla u_n -\nabla u]$ is easily absorbed by the left-hand side of the inequality, and the key observation is that the remaining term $[g'(u_n) -g'(u)] \nabla u$ is of lower order, in the sense that the convergence of $u_n $ to $u$ in appropriate $L^p $ spaces implies that this last term converges to $0$.
We prove Theorem~\ref{eMain} by applying the same idea. The key argument of the proof is an estimate which shows that  $g(u)-g(v)$ is bounded in Besov spaces by a Lipschitz term (i.e. with a factor $u-v$) plus some lower order term. (See Lemmas~\ref{eLemLocLip} and~\ref{eLemLip} below.)
In the critical case $\alpha = 4/(N-2s)$, a further argument is required in order to show that $u_n $ to $u$ in the appropriate $L^p $ space (estimate~\eqref{fPrelimcb}). For this, following Tao and Visan~\cite{TaoV}, we use a Strichartz-type estimate for a non-admissible pair, and this is where
we use the assumption~\eqref{fSbCr}.
Local Lipschitz continuity in Corollary~\ref{eMaind} follows from a different, much simpler argument: in this case, the nonlinearity is locally Lipschitz in the appropriate Besov spaces.

In Theorem~\ref{eMain} we assume  $s<\min \{ 1, N/2\}$.
The assumption $s<N/2$ is natural, but the restriction $s<1$ is technical.
When $N\ge 3$, local existence in $H^s(\R^N )$ is known to hold for $0<s<N/2$,
in particular when $g(u)= \lambda  |u|^\alpha u$, under some extra assumption on $\alpha $ and $s$ that ensure that $g$ is sufficiently smooth. See~\cite{CWHs, Katou, Pecher}.
The limitation $s<1$ first appears in our Besov space estimates. These could possibly be extended to $s>1$. It also appears in a more subtle way. For example, the introduction of certain exponents~\eqref{fRDu} in the critical case explicitly requires $s<1$.
It is not impossible that the idea in~\cite{TaoV} of using intermediate Besov spaces of lower order can be applied when $s>1$.

We next mention a few open questions.
As observed above, the fixed point argument used in~\cite{CWHs,Katou,CLN} to prove Theorem~\ref{echysecpzu} does not show that the flow is locally Lipschitz in $H^s (\R^N ) $.
On the other hand, it does not show either that it is not locally Lipschitz.
This raises the following question: under the assumptions~\eqref{fHypu}--\eqref{fHypq}, is the flow of~\eqref{NLS} locally Lipschitz in $H^s$?
We suspect that the answer might be negative.
Note that in the pure power case $g(u)= \lambda  |u|^\alpha u$, Remark~\ref{eLemLocLipd} implies that the nonlinearity is locally H\"older continuous in the appropriate Besov spaces of order $s$.
Therefore it is natural to ask if the flow also is locally H\"older continuous in $H^s$.
Finally, we note that in the critical case $\alpha =\frac {4} {N-2s}$ Theorem~\ref{eMain}
imposes the restriction $A=0$ in~\eqref{fHypt}. Can this restriction be removed?

The rest of this paper is organized as follows. In Section~\ref{Besov} we establish estimates of $g(u)-g(v)$ in Besov spaces.
We complete the proof Theorem~\ref{eMain} in Section~\ref{Proof} in the subcritical case
$\alpha <\frac {4} {N-2s}$
and in Section~\ref{ProofCr} in the critical case $\alpha =\frac {4} {N-2s}$.
Section~\ref{ProofPwr} is devoted to the proof of Corollary~\ref{eMaind}.

\medskip
\noindent {\bf Notation.}\quad
Given $1\le p\le \infty $, we
denote by $p'$ its conjugate given by $1/p'= 1- 1/p$
and we consider the standard (complex-valued) Lebesgue spaces $L^p(\R^N ) $.
Given $1\le p, q\le \infty $ and $s \in \R$ we consider the usual (complex-valued) Sobolev and Besov spaces  $H^s(\R^N) $ and $B^s _{ p,q }(\R^N )$
and their homogeneous versions $\dot H^s(\R^N) $ and $\dot B^s _{ p,q }(\R^N )$.
See for example~\cite{BerghL, Triebel} for the definitions of these spaces and the
corresponding Sobolev's embeddings.
We denote by $(e^{it\Delta }) _{ t\in \R }$ the Schr\"o\-din\-ger group and
we will use without further reference the standard Strichartz estimates, see~\cite{Strichartz,
Yajima, Katot, KeelT}.

\section{Superposition operators in Besov spaces} \label{Besov}

The study of superposition operators in Besov spaces has a long history and
necessary conditions (sometimes necessary and sufficient conditions) on $g$ are known so that  $u\mapsto g(u)$ is a bounded map of certain Besov spaces. (See e.g.~\cite{RunstS} and the references therein.) On the other hand, few works study the continuity of such maps, among which~\cite{BourdaudL, BourdaudLS, BourdaudLSd, BourdaudMS, BrezisM},
but none of these works applies to such cases as $g(u)=  |u|^\alpha u$, which is a typical nonlinearity for~\eqref{NLS}. Besides,
 for our proof of Theorem~\ref{eMain} we do not need only continuity but instead a rather specific property, namely that $g(u)-g(v)$ can be estimated in a certain Besov space of order $s$ by a Lipschitz term (i.e. involving $u-v$ in some other Besov space of order $s$) plus a lower order term. We establish such estimates in the following two lemmas. The first one concerns functions $g$ such that $ |g'(u)|\le C  |u|^\alpha $, and the second functions $g$ that are globally Lipschitz.
(A function $g$ satisfying~\eqref{fHypd}-\eqref{fHypt} can be decomposed as the sum of two such functions.)
The proofs of these lemmas are  based on the
choice of an equivalent norm on $\dot B^s  _{ p,q }$ defined in terms of finite differences, and on rather elementary calculations.

\begin{lem} \label{eLemLocLip}
Let $0<s<1$, $\alpha >0$, $1\le q<\infty $ and $1\le p<r\le \infty $, and let $0<\sigma <\infty $ be defined by
\begin{equation} \label{fDfnSig}
\frac {\alpha } {\sigma }= \frac {1} {p}- \frac {1} {r}.
\end{equation}
We use the convention that $ \|u\| _{ L^\sigma  } = (\int  _{ \R^N  }  |u|^\sigma )^{\frac {1} {\sigma }}$
even if $\sigma <1$.
Let $g\in C^1(\C, \C)$ satisfy
\begin{equation}  \label{eLemLocLip:u}
 |g'(u)| \le C  |u|^\alpha ,
\end{equation}
for all $u\in \C$.
It follows that there exists a constant $C$ such that if $u,v\in \dot B^s  _{ r,q } (\R^N )$
and $ \|u\| _{ L^\sigma  }, \|v\| _{ L^\sigma  } <\infty $, then
\begin{equation}\label{eLemLocLip:d}
 \|g(v)- g(u)\| _{ \dot B^s  _{ p,q } } \le C \|v\| _{ L^\sigma  }^\alpha
   \|v-u\| _{ \dot B^s  _{ r,q } }  + K(u,v),
\end{equation}
where $K(u,v)$ satisfies
\begin{equation}\label{eLemLocLip:t}
 K(u,v) \le C ( \|u\| _{ L^\sigma  }^\alpha + \|v\| _{ L^\sigma  }^\alpha )  \|u\| _{ \dot B^s  _{ p,q } },
\end{equation}
and
\begin{equation}\label{eLemLocLip:q}
K(u, u_n) \goto  _{ n\to \infty  }0,
\end{equation}
if $(u_n) _{ n\ge 1 } \subset \dot B^s  _{ r,q } (\R^N ) $ and
$u\in  \dot B^s  _{r,q } (\R^N ) $ are such that $ \|u\| _{ L^\sigma  }  <\infty $
 and $ \|u_n -u\| _{ L^\sigma  } \to 0$ as $n\to \infty $.
\end{lem}

\begin{proof}
We denote by $\tau _y$ the translation operator defined for $y\in \R^N $ by
$\tau _y u(\cdot )= u(\cdot -y)$ and
we recall that (see e.g. Section~5.2.3, Theorem~2, p.~242 in~\cite{Triebel})
\begin{equation} \label{fLemLipu}
 \|u\| _{ \dot B^s  _{ p,q } }  \approx \Bigl(\int  _{  \R^N   }  \|\tau _y u -u\|_{L^p (\R^N )}^q  |y|^{-N- sq}
dy\Bigr)^{\frac {1} {q}} .
\end{equation}
Given $z_1, z_2\in \C$, we have
\begin{multline*}
g(z_1)- g(z_2)= (z_1-z_2) \int _0^1 \partial _z g (z_2+\theta (z_1-z_2)) \,d\theta + \\
( \overline{z_1-z_2} ) \int _0^1 \partial _{ \overline{z} }g (z_2+\theta (z_1-z_2)) \,d\theta,
\end{multline*}
which we write, in short,
\begin{equation} \label{fLemLipd}
g(z_1)- g(z_2)= (z_1-z_2) \int _0^1 g' (z_2+\theta (z_1-z_2)) \,d\theta.
\end{equation}
We set
\begin{equation} \label{fLemLipt}
A(u,v,y) (\cdot )= \tau _y [ g(v)- g(u) ]- [g(v)- g(u)],
\end{equation}
so that by~\eqref{fLemLipu}
\begin{equation} \label{fLemLipq}
 \| g(v)- g(u)\| _{ \dot B^s _{ p,q } } \le C
 \Bigl(\int  _{  \R^N   }  \|A(u,v,y)\|_{L^p}^q  |y|^{-N- sq}
dy\Bigr)^{\frac {1} {q}} .
\end{equation}
We deduce from~\eqref{fLemLipt} and~\eqref{fLemLipd} that
\begin{multline*}
A (u,v,y)=
[  g(\tau _y v)- g(v)] - [  g(\tau _y u)- g(u)]
\\ =  (\tau _y v-v) \int _0^1 g' (v +\theta ( \tau _y v -v)) \,d\theta
-  (\tau _y u -u ) \int _0^1 g' (u+\theta (\tau _y u-u)) \,d\theta ;
\end{multline*}
and so
\begin{multline} \label{fLemLipc}
A(u,v,y) =  [\tau _y (v-u) -(v-u)] \int _0^1 g' (v +\theta ( \tau _y v -v)) \,d\theta \\
+   (\tau _y u -u )  \int _0^1 [g' (v+\theta (\tau _y v-v))
-  g' (u+\theta (\tau _y u-u)) ]\,d\theta
  \\ = : A_1(u,v,y) + A_2(u,v,y).
\end{multline}
It follows from~\eqref{eLemLocLip:u} that
\begin{equation*}
 |A_1(u,v,y)| \le C ( |v|^\alpha  +  |\tau _y v|^\alpha )  | \tau _y (v-u) -(v-u) |.
\end{equation*}
We deduce by H\"older's inequality (note that $\sigma /\alpha \ge 1$) that
\begin{equation*}
 \|A_1(u,v,y)\| _{ L^p } \le C  \|v\| _{ L^\sigma  }^\alpha   \|\tau _y (v-u) -(v-u)\| _{ L^r },
\end{equation*}
from which it follows by applying~\eqref{fLemLipu} that
\begin{equation} \label{fLemLocLipu}
 \Bigl(\int  _{  \R^N   }  \|A_1(u,v,y)\|_{L^p }^q  |y|^{-N- sq}
dy\Bigr)^{\frac {1} {q}}  \le C  \|v\| _{ L^\sigma  }^\alpha      \|v-u\| _{ \dot B^s  _{ r,q } }.
\end{equation}
We set
\begin{equation} \label{fLemLiph}
K(u,v)=  \Bigl(\int  _{  \R^N   }  \|A_2(u,v,y)\|_{L^p}^q  |y|^{-N- sq}
dy\Bigr)^{\frac {1} {q}},
\end{equation}
and formula~\eqref{eLemLocLip:d}  follows from~\eqref{fLemLipq}, \eqref{fLemLipc},
\eqref{fLemLocLipu} and~\eqref{fLemLiph}.
Next, it follows from assumption~\eqref{eLemLocLip:u} that
\begin{equation} \label{fLemLocLipub}
 |A_2(u,v,y)| \le C ( |u|^\alpha  +  |\tau _y u|^\alpha + |v|^\alpha  +  |\tau _y v|^\alpha )  | \tau _y u -
 u|.
\end{equation}
Applying H\"older's inequality, we deduce as above that
\begin{equation} \label{fLemLocLipd}
 \Bigl(\int  _{  \R^N   }  \|A_2(u,v,y)\|_{L^p }^q  |y|^{-N- sq}
dy\Bigr)^{\frac {1} {q}}  \le C(  \|u\| _{ L^\sigma  }^\alpha  +  \|v\| _{ L^\sigma  }^\alpha  )    \|u\| _{ \dot B^s  _{ r,q } },
\end{equation}
which proves~\eqref{eLemLocLip:t}.  It remains to show~\eqref{eLemLocLip:q}.
Let $(u_n) _{ n\ge 1 }$ and $u$ be as in the statement, and assume by contradiction that
 there is a subsequence, still denoted by $(u_n) _{ n\ge 1 }$, and $\varepsilon >0$ such that
\begin{equation}  \label{fLemLocLipt}
K(u, u_n) \ge \varepsilon .
\end{equation}
We set
\begin{equation} \label{fLemLipud}
 B_n (y,\theta )  (\cdot )= |\tau _y u -u |  |g' (u_n+\theta (\tau _y u_n-u_n))
-  g' (u+\theta (\tau _y u-u))|,
\end{equation}
and we deduce from~\eqref{eLemLocLip:u} that
\begin{equation}  \label{fLemLocLipq}
 | B_n (y,\theta )| \le C ( |u|^\alpha +  |\tau _y u|^\alpha + |u_n|^\alpha
 +  |\tau _y u_n|^\alpha )  |\tau _y u -u | .
\end{equation}
It follows from~\eqref{fLemLipc}   that
\begin{equation} \label{fLemLocLipc}
 \|A_2 (u,u_n,y)\| _{ L^p }^p \le \int _0^1 \int  _{ \R^N  }
 B_n (y, \theta ) ^p
  \,dx \,d\theta .
\end{equation}
Note that $ |u_n-u|^\sigma \to 0$ in $L^1 (\R^N ) $. In particular, by possibly extracting a subsequence, we see that $u_n \to u$ as $n\to \infty $ a.e. and that there exists a function $w\in L^1 (\R^N ) $ such that $|u_n-u|^\sigma \le w$ a.e. It follows that $ |u_n|^\sigma \le C ( |u|^\sigma + w) \in L^1 (\R^N ) $.
Moreover, by Young's inequality
\begin{equation}  \label{fLemLocLips}
 B_n (y,\theta ) ^p\le C ( |u|^\sigma + |\tau _yu|^\sigma  + |u_n|^\sigma + |\tau _y u_n|^\sigma +
 |\tau _y u -u|^r).
\end{equation}
Since by~\eqref{fLemLipu} $\tau _y u -u\in L^r(\R^N ) $ for a.a. $y\in \R^N $,
we see that for a.a. $y\in \R^N $, $B_n (y,\theta ) ^p$ is dominated by an $L^1$ function.
Furthermore, $g'$ is continuous and $u_n\to u$ a.e., so that $B_n (y,\theta ) \to 0$ a.e. as $n\to \infty $. We deduce by dominated convergence and~\eqref{fLemLocLipc}  that
\begin{equation}  \label{fLemLocLipp}
\|A_2 (u,u_n,y)\| _{ L^p }^p \goto _{ n\to \infty  }0,
\end{equation}
for almost all $y\in \R^N $.
Next, it follows from~\eqref{fLemLocLipub} and H\"older's inequality that
\begin{equation*}
 \|A_2 (u,u_n,y)\| _{ L^p }\le C
 ( \|u\| _{ L^\sigma  }^\alpha +  \|u_n\| _{ L^\sigma  }^\alpha )  \|\tau _y u-u\| _{ L^r }
 \le C  \|\tau _y u-u\| _{ L^r }.
\end{equation*}
Since
\begin{equation*}
 \|\tau _y u -u \|_{L^r }^q  |y|^{-N- sq} \in L^1 (\R^N ) ,
\end{equation*}
by~\eqref{fLemLipu},
we see that $ \|A_2 (u,u_n,y)\| _{ L^p }^q  |y|^{-N- sq}$ is dominated by an $L^1$ function.
  Applying~\eqref{fLemLocLipp}, we deduce by dominated convergence that
$K(u, u_n) \to 0$ as $n\to \infty $. This contradicts~\eqref{fLemLocLipt} and completes the proof.
\end{proof}

\begin{lem} \label{eLemLip}
Let $0<s<1$ and $1\le p\le \infty $, $1\le q<\infty $.
Let $g\in C^1(\C, \C)$ with $g'$ bounded.
It follows that there exists a constant $C$ such that if $u,v\in \dot B^s  _{ p,q } (\R^N )$, then
\begin{equation} \label{eLemLip:u}
 \|g(v)- g(u)\| _{ \dot B^s  _{ p,q } } \le C  \|v-u\| _{ \dot B^s  _{ p,q } }  + K(u,v),
\end{equation}
where $K(u,v)$ satisfies
\begin{equation} \label{eLemLip:d}
 K(u,v) \le C  \|u\| _{ \dot B^s  _{ p,q } },
\end{equation}
and
\begin{equation} \label{eLemLip:t}
K(u, u_n) \goto  _{ n\to \infty  }0,
\end{equation}
if $(u_n) _{ n\ge 1 } \subset \dot B^s  _{ p,q } (\R^N )  $ and
$u\in  \dot B^s  _{ p,q } (\R^N )  $ are such that  $u_n \to u$ in $L^\mu   (\R^N )$ for some $1\le \mu  \le \infty $.
\end{lem}

\begin{proof}
The proof of Lemma~\ref{eLemLocLip} is easily adapted.
\end{proof}

\begin{rem} \label{eLemLocLipd}
In the particular case when $g(u) =  |u|^\alpha u$ with $\alpha >0$,  the estimate~\eqref{eLemLocLip:d} can be refined. More precisely,
\begin{equation} \label{feRemBSu}
 \|g(v)- g(u)\| _{ \dot B^s  _{ p,q } } \le C \|v\| _{ L^\sigma  }^\alpha
   \|v-u\| _{ \dot B^s  _{ r,q } }  + C  \|u\| _{ \dot B^s  _{ r,q } }
    \|u-v\| _{ L^\sigma  }^\alpha ,
\end{equation}
if $0<\alpha \le 1$ and
\begin{multline} \label{feRemBSd}
 \|g(v)- g(u)\| _{ \dot B^s  _{ p,q } } \le C \|v\| _{ L^\sigma  }^\alpha
   \|v-u\| _{ \dot B^s  _{ r,q } } \\ + C  \|u\| _{ \dot B^s  _{ r,q } }
   ( \|u\| _{ L^\sigma  }^{\alpha -1} + \|v\| _{ L^\sigma  }^{\alpha -1})
    \|u-v\| _{ L^\sigma  },
\end{multline}
if $\alpha \ge 1$.
Indeed, note that $\partial _z g(z)= (1+ \frac {\alpha } {2})  |z|^\alpha $ and $\partial  _{  \overline{z}  }g(z)= \frac {\alpha  } {2}  |z|^{\alpha -2} z^2$.
We claim that,
\begin{equation} \label{feRemBSt}
 |\,  |z_1|^\alpha - |z_2|^\alpha  |\le
 \begin{cases}
 \alpha ( |z_1|^{\alpha -1} +  |z_2|^{\alpha -1})  |z_1-z_2|
 & \text{if }\alpha \ge 1, \\
  |z_1-z_2|^\alpha &  \text{if }0<\alpha \le 1,
 \end{cases}
\end{equation}
and
\begin{equation} \label{feRemBSq}
 |\,  |z_1|^{\alpha -2} z_1^2 -  |z_2|^{\alpha -2} z_2^2  |\le
 \begin{cases}
  C ( |z_1|^{\alpha -1} +  |z_2|^{\alpha -1})   |z_1-z_2| & \text{if }\alpha \ge 1, \\
C  |z_1-z_2|^\alpha &  \text{if }0<\alpha \le 1.
 \end{cases}
\end{equation}
Assuming~\eqref{feRemBSt}-\eqref{feRemBSq},
estimates~\eqref{feRemBSu} and~\eqref{feRemBSd} follow from the argument used at the beginning of the proof of Lemma~\ref{eLemLocLip}. Indeed, the left-hand side of~\eqref{fLemLocLipd} (where $A_2$ is defined by~\eqref{fLemLipc})
  is easily estimated by applying~\eqref{feRemBSt}-\eqref{feRemBSq} and H\"older's inequality.
It remains to prove the claim~\eqref{feRemBSt}-\eqref{feRemBSq}.
The first three estimates are quite standard, and we only prove the last one.
We assume, without loss of generality, that $0< |z_2|\le  |z_1|$.
Note first that
\begin{equation*}
\begin{split}
|\,  |z_1|^{\alpha -2} z_1^2 -  |z_2|^{\alpha -2} z_2^2 | & =
\Bigl|    |z_1|^{\alpha }  \Bigl( \frac {z_1^2} { |z_1|^2} - \frac {z_2^2} { |z_2|^2}  \Bigr)
+ \frac {z_2^2} { |z_2|^2}  ( |z_1|^\alpha -  |z_2|^\alpha ) \Bigr|
\\ & \le |z_1|^{\alpha }  \Bigl| \frac {z_1^2} { |z_1|^2} - \frac {z_2^2} { |z_2|^2}  \Bigr|
+  |z_1- z_2|^\alpha .
\end{split}
\end{equation*}
Next,
\begin{multline*}
 |z_1|^{\alpha }  \Bigl| \frac {z_1^2} { |z_1|^2} - \frac {z_2^2} { |z_2|^2}  \Bigr|
  \le 2  |z_1|^{\alpha }  \Bigl| \frac {z_1} { |z_1|} - \frac {z_2} { |z_2|}  \Bigr|
 \\  = 2  |z_1|^{\alpha -1 }  \Bigl| z_1-z_2 +
   \frac {z_2} { |z_2|} ( |z_2| - |z_1|)  \Bigr|
 \le 4   |z_1|^{\alpha -1 }   |z_1- z_2|.
\end{multline*}
If $ |z_1-z_2| \le  |z_1|$, then (since $\alpha \le 1$) $ |z_1|^{\alpha -1 }   |z_1- z_2| \le |z_1- z_2|^\alpha $. If $ |z_1-z_2| \ge  |z_1|$, then (recall that $ |z_1-z_2|\le  |z_1|+  |z_2|\le 2 |z_1|$)
we see that $ |z_1|^{\alpha -1 }   |z_1- z_2| \le 2 |z_1|^\alpha \le 2 |z_1-z_2|^\alpha $. This shows the second estimate in~\eqref{feRemBSq}.
\end{rem}

\section{Proof of Theorem~$\ref{eMain}$
in the subcritical case $\alpha <\frac {4} {N-2s}$} \label{Proof}

Throughout this section, we assume $\alpha <\frac {4} {N-2s}$ and we use the admissible pair
$(\gamma ,\rho )$ defined by~\eqref{fchysecpzt}.
We need only show that, given $\varphi \in H^s (\R^N ) $, there exists $T = T( \| \varphi \| _{ H^s })>0$   such that if $ \|\varphi \| _{ H^s }< M$ then $\Tma (\varphi )>T$ and such that if $\varphi _n\to \varphi $ in $H^s (\R^N ) $, then $\Tma (\varphi _n)>T$ for all sufficiently large $n$ and
 the corresponding solution $u_n$ of~\eqref{NLS} satisfies $u_n\to u$ in
$L^q ((0, T),B^s_{ r,2} (\R^N ))$ for all admissible pair $(q,r)$.
Since $T= T( \|\varphi \| _{ H^s })$, properties~\eqref{eMain:u} and~\eqref{eMain:d} easily follow by a standard iteration argument.

We note that by Theorem~\ref{echysecpzu} there exists $T=T( \| \varphi \| _{ H^s })>0$ such that if $ \|\varphi \| _{ H^s }<M$ then $\Tma (\varphi )>T$.
Moreover, by possibly choosing $T$ smaller (but still depending on $ \|\varphi \| _{ H^s }$), if $\varphi _n\to \varphi $ in $H^s (\R^N ) $, then
$\Tma (\varphi _n)>T$ for all sufficiently large $n$ and the corresponding solution $u_n$ of~\eqref{NLS} satisfies
\begin{equation} \label{fConvPu}
\sup  _{ n\ge 1 }  \|u_n\| _{ L^q((0,T), B^s _{ r,2 }) }< \infty ,
\end{equation}
and
\begin{equation} \label{fConvPub}
u_n \goto _{ n\to \infty  }u  \text{ in } L^{q}((0,T), B^{s-\varepsilon } _{ r ,2 } (\R^N ) )
\cap C([0,T], H^{s-\varepsilon } (\R^N ) ),
\end{equation}
for all $\varepsilon >0$ and all admissible pair $(q,r)$.
It is not specified in Theorem~\ref{echysecpzu} that,
in the subcritical case, $T$ can be bounded from below in terms of $ \|\varphi \| _{ H^s }$ but this is immediate from the proof.

We set
\begin{equation} \label{fSuplt}
\sigma = \frac {N(\alpha +2)} {N-2s} >\rho ,
\end{equation}
so that by Sobolev's embedding
\begin{equation}  \label{fSupltb}
\dot B^s  _{ \rho ,2 } (\R^N ) \hookrightarrow L^\sigma  (\R^N ) ,
\end{equation}
and we claim that
\begin{equation} \label{fZsplu}
u_n \goto _{ n\to \infty  }u  \text{ in } L^{\gamma -\varepsilon } ((0,T),
L^\sigma  (\R^N ) ),
\end{equation}
for all $0<\varepsilon \le \gamma -1$.
This is a consequence of~\eqref{fConvPub}. (In fact,
if we could let $\varepsilon =0$ in~\eqref{fConvPub}, then we would obtain by~\eqref{fSupltb} convergence in  $L^{\gamma } ((0,T),  L^\sigma  (\R^N ) )$.)
Indeed, given $\eta >0$ and small, we have $B^{s- \frac {\eta N} {\rho (\rho +\eta)} } _{ \rho +\eta,2 } (\R^N ) \hookrightarrow L^\sigma  (\R^N ) $.
If $\eta$ is sufficiently small, $\rho +\eta <2N/ (N-2)^+$ so that there exists
 $\gamma _\eta $ such that $(\gamma _\eta ,\rho +\eta)$ is an admissible pair,
 and we deduce from~\eqref{fConvPub} that $u_n \to u$ in $L^{\gamma _\eta }((0,T), L^\sigma (\R^N ))$.   Since $\gamma _\eta \to \gamma $ as $\eta \to 0$, we conclude that~\eqref{fZsplu} holds.

We  decompose $g$ in the form $g= g_1 +g_2$ where $g_1, g_1 \in C^1(\C, \C)$, $g_1(0)= g_2(0)= 0$ and
\begin{gather}  \label{fConvPs}
 |g'_1  (u)|\le C , \\
  |g_2' (u)|\le C  |u|^\alpha ,
\end{gather}
for all $u\in \C$. By Strichartz estimates in Besov spaces (see Theorem~2.2 in~\cite{CWHs}),
given any admissible pair $(q,r)$ there exists a constant $C$ such that
\begin{multline} \label{fConvPp}
\|u_n-u \| _{ L^q ((0,T), \dot B^s _{ r ,2 } )}
 \le C  \|\varphi _n-\varphi \| _{ \dot H^s } \\ +  C\| g_1(u_n)- g_1(u)\| _{ L^{ 1} ((0, T), \dot B^s  _{ 2,2}) }
  +  C\| g_2(u_n)- g_2(u)\| _{ L^{\gamma '} ((0, T), \dot B^s  _{ \rho ' ,2}) }.
\end{multline}
We estimate the last two terms in~\eqref{fConvPp}
by applying Lemmas~\ref{eLemLip} and~\ref{eLemLocLip}, respectively.
We first apply Lemma~\ref{eLemLip}  to $g_1$ with $q=p=2$ and we obtain
\begin{equation} \label{fSuplu}
 \|g_1 (u_n)-g_1 (u)\| _{ \dot B^s  _{ 2,2 } }\le C  \|u_n -u \| _{ \dot B^s  _{ 2,2 } }
 + K_1(u, u_n).
\end{equation}
Next, we apply Lemma~\ref{eLemLocLip} to $g_2$ with $q=2$, $r=\rho $ and $p=\rho '$,
and we obtain
\begin{equation} \label{fSupld}
 \|g_2 (u_n)-g_2(u)\| _{ \dot B^s  _{ \rho ' ,2 } }\le C
  \|u_n\| _{ L^\sigma  }^\alpha  \|u_n -u \| _{ \dot B^s  _{ \rho ,2 } }
 + K_2(u, u_n),
\end{equation}
where $\sigma $ is defined by~\eqref{fSuplt}.
Applying H\"older's inequality in time, we deduce from~\eqref{fSuplu} that
\begin{multline} \label{fSuplq}
\|g_1 (u_n)- g_1 (u)\|  _{ L^{1}((0,T), \dot B^s_{2,2} ) } \\ \le C
T   \|u_n-u\| _{L^{\infty } ((0,T), \dot B^s_{2,2}) }
 +  \|K_1(u, u_n)\| _{ L^{1}(0,T) },
\end{multline}
and from~\eqref{fSupld}  that
\begin{multline} \label{fSuplc}
\|g_2 (u_n)- g_2 (u)\|  _{ L^{\gamma '}((0,T), \dot B^s_{\rho ',2} ) } \\  \le C
T^{\frac {4- \alpha (N-2s)} {4}}
\|u_n \|  _{ L^{\gamma } ((0,T), L^{\sigma } ) }^\alpha
 \|u_n-u\| _{L^{\gamma } ((0,T), \dot B^s_{\rho ,2}) }  \\
 +  \|K_2(u,u_n)\| _{ L^{\gamma '}(0,T) }.
\end{multline}
Note that  $ \|u_n \|  _{ L^{\gamma } ((0,T), L^{\sigma } ) }^\alpha $ is bounded
 by~\eqref{fConvPu} and~\eqref{fSupltb}. Thus we see that, by possibly choosing $T$ smaller (but still depending on $ \|\varphi \| _{ H^s }$), we can absorb the first term in the right hand side of~\eqref{fSuplq} by the left hand side of~\eqref{fConvPp} (with the choice $(q,r)= (\infty ,2)$),
 and similarly for~\eqref{fSuplc} (with the choice $(q,r)= (\gamma  ,\rho )$).
 By doing so, we deduce from~\eqref{fConvPp} that
\begin{multline} \label{fConvPpb}
 \|u_n-u \| _{ L^\gamma ((0,T), \dot B^s _{ \rho ,2 } )}+  \|u_n -u \| _{ L^\infty ((0,T), \dot H^s) }
\\  \le C  \|\varphi _n-\varphi \| _{ \dot H^s } +  C \|K_1(u, u_n)\| _{ L^{1}(0,T) } +  C \|K_2(u,u_n)\| _{ L^{\gamma '}(0,T) }.
\end{multline}
Next, we  note that by~\eqref{eLemLip:d} $K_1 (u, u_n)\le  \|u (t)\| _{ \dot B^s _{ 2,2 } }$ for all $t\in (0,T)$ and that $u\in L^\infty ((0,T), B^s _{ 2,2 } (\R^N ))$.
Moreover, $u_n\to u$ in $C([0,T], L^2 (\R^N ) )$, so that $u_n(t)\to u(t)$ in $L^2 (\R^N ) $ for all $t\in (0,T)$. Applying~\eqref{eLemLip:t} (with $\mu  =2$),  we deduce that $K_1(u_n, u) \to 0$ for all $t\in (0,T)$. By dominated convergence, we conclude that
\begin{equation} \label{fSuplp}
 \|K_1(u, u_n)\| _{ L^{1}(0,T) }  \goto _{ n\to \infty  }0.
\end{equation}
We now show that
\begin{equation} \label{fSuplh}
 \|K_2(u, u_n)\| _{ L^{\gamma '}(0,T) }  \goto _{ n\to \infty  }0.
\end{equation}
Indeed, suppose  by contradiction that there exist a subsequence, still denoted by $(u_n) _{ n\ge 1 }$, and $\delta  >0$ such that
\begin{equation} \label{fSupln}
 \|K_2(u, u_n)\| _{ L^{\gamma '}(0,T) } \ge \delta  .
\end{equation}
We note that by~\eqref{eLemLocLip:t} and Young's inequality
\begin{equation} \label{fSPu}
\begin{split}
K_2 (u_n, u)^{\gamma '} &\le C ( \|u_n \| _{ L^\sigma  }^{\alpha \gamma '}+  \|u \| _{ L^\sigma  }^{\alpha \gamma '})  \|u\| _{ \dot B^s _{ \rho ,2 }}^{\gamma ' }  \\&  \le
C(  \|u_n \| _{ L^\sigma  }^{\frac {\alpha \gamma } {\gamma -2}}+
  \|u \| _{ L^\sigma  }^{\frac {\alpha \gamma } {\gamma -2}}+   \|u\| _{ \dot B^s _{ \rho ,2 }}^{\gamma  } ) .
\end{split}
\end{equation}
Since $\alpha \gamma /(\gamma -2) <\gamma $, we deduce from~\eqref{fZsplu} that, after possibly extracting a subsequence, $K_2 (u_n, u)^{\gamma '} $ is dominated by an $L^1$ function.
It also follows from~\eqref{fZsplu} that, after possibly extracting a subsequence,
$u_n(t) \to u(t)$ in $L^\sigma  (\R^N ) $ for a.a. $t\in (0,T)$ so that, applying~\eqref{eLemLocLip:q},
\begin{equation} \label{fSPd}
K _2(u, u_n) \goto _{ n\to \infty  }0,
\end{equation}
for a.a. $t\in (0,T)$. By dominated convergence, $ \|K_2(u, u_n)\| _{ L^{\gamma '}(0,T) }  \to 0$. This contradicts~\eqref{fSupln}, thus proving~\eqref{fSuplh}.

Finally, it follows from~\eqref{fConvPpb}, \eqref{fSuplp} and~\eqref{fSuplh}
that
\begin{equation} \label{fLstih}
 \|u_n-u \| _{ L^\gamma ((0,T), \dot B^s _{ \rho ,2 } )}+  \|u_n -u \| _{ L^\infty ((0,T), \dot H^s) } \goto  _{ n\to \infty  } 0.
\end{equation}
Given any admissible pair $(q,r)$, we deduce from~\eqref{fConvPp}, \eqref{fSuplq}, \eqref{fSuplc}, \eqref{fSuplp}, \eqref{fSuplh} and~\eqref{fLstih}
that    $ \|u_n-u \| _{ L^q ((0,T), \dot B^s _{ r ,2 } )} \to 0$ as $n\to \infty $. This completes the proof.

\section{Proof of Theorem~$\ref{eMain}$
in the critical case $\alpha =\frac {4} {N-2s}$} \label{ProofCr}

Throughout this section, we assume $\alpha =\frac {4} {N-2s}$ and we use the admissible pair
$(\gamma ,\rho )$ defined by~\eqref{fchysecpzt}.
We recall that by assumption~\eqref{fSbCr},
\begin{equation} \label{fCritub}
 |g'(u)|\le C  |u|^\alpha .
\end{equation}
The argument of the preceding section fails essentially at one point. More precisely,
here $\alpha \gamma /(\gamma -2) = \gamma $, so that the convergence~\eqref{fZsplu} does not imply as in Section~\ref{Proof} (see~\eqref{fSPu})  that $K_2 (u_n, u)^{\gamma '} $ is dominated by an $L^1$ function.
To overcome this difficulty, we show in Lemma~\ref{eKey} below that~\eqref{fZsplu} holds with $\varepsilon =0$. This is the key difference with the subcritical case,
and the rest of the proof is very similar.

We now go into details and we first recall some facts concerning the local theory.
(See e.g.~\cite{CWHs}.)
We define
\begin{equation} \label{fDefng}
{\mathcal G}(u) (t)= \int _0^t e^{i(t-s)\Delta } g(u(s))\, ds.
\end{equation}
If $(\gamma ,\rho )$ is the admissible pair defined by~\eqref{fchysecpzt}, then
\begin{equation*}
 \|g(u)\| _{ L^{\gamma '}((0,T), B^s _{ \rho ',2 } )} \le  C\| u\| _{ L^\gamma ((0,T), B^s _{ \rho ,2 } )}^{\alpha +1},
\end{equation*}
and
\begin{multline*}
 \|g(v)- g(u)\| _{ L^{\gamma '}((0,T), L^{\rho '} )} \\
  \le  C(\| u\| _{ L^\gamma ((0,T), B^s _{ \rho ,2 } )}^{\alpha }+ \| u\| _{ L^\gamma ((0,T), B^s _{ \rho ,2 } )}^{\alpha })  \|u-v\| _{ L^\gamma ((0,T), L^\rho ) },
\end{multline*}
where the constant $C$ is independent of $T>0$.
By Strichartz estimates in Besov spaces, it follows that the problem~\eqref{INLS}  can be solved by a fixed point argument in the set ${\mathcal E}=  \{u\in L^\gamma ((0,T), B^s _{ \rho ,2 } (\R^N ));
 \|u\| _{ L^\gamma ((0,T), B^s _{ \rho ,2 } )}\le \eta\}$ equipped with the distance ${\mathrm d}(u,v)=  \|u-v\|_{ L^\gamma ((0,T), L^\rho ) }$ for $\eta>0$ sufficiently small.
(Note that $({\mathcal E}, {\mathrm d})$
is a complete metric space.  Indeed, $L^\gamma ((0,T),
B^s_{\rho ,2})$ is reflexive, so its closed ball of radius $\eta$ is weakly
compact.)
By doing so, we see that there exists $\delta _0>0$
(independent of $\varphi \in H^s (\R^N ) $ and $T>0$) such that if
\begin{equation} \label{fPrelimu}
 \| e^{i\cdot  \Delta } \varphi \|_{ L^\gamma ((0,T), B^s _{ \rho ,2 }) } < \delta \le \delta _0,
\end{equation}
then $\Tma (\varphi )>T$ and the corresponding solution of~\eqref{INLS} satisfies
\begin{equation} \label{fPrelimd}
 \| u \|_{ L^\gamma ((0,T), B^s _{ \rho ,2 }) } \le 2 \delta.
\end{equation}
Moreover, given any admissible pair $(q,r)$, there exists a constant $C(q,r)$ such that
\begin{equation} \label{fPrelimt}
 \| u \|_{ L^q ((0,T), B^s _{ r ,2 }) } \le \delta C(q,r).
\end{equation}
In addition, if $\varphi,\psi $ both satisfy~\eqref{fPrelimu} and $u,v$ are the corresponding solutions of~\eqref{NLS}, then
\begin{equation} \label{fPrelimq}
 \| u -v\|_{ L^\gamma  ((0,T),L^\rho ) } \le C  \|\varphi -\psi \| _{ L^2 }.
\end{equation}
We also recall the Strichartz estimate in Besov spaces
\begin{equation} \label{fSBesu}
\| e^{i\cdot  \Delta }  \varphi \|_{ L^q((0,T), B^s _{ r,2 }) } \le C(q,r) \|\varphi \| _{ H^s },
\end{equation}
for every admissible pair $(q,r)$.  In particular,
\begin{equation} \label{fSplzu}
\| e^{i\cdot  \Delta }  \psi  \|_{ L^\gamma ((0,T), B^s _{ \rho ,2 }) } \le
\| e^{i\cdot  \Delta }  \varphi   \|_{ L^\gamma ((0,T), B^s _{ \rho ,2 }) } +
C \|\psi -\varphi \| _{ H^s },
\end{equation}
for every $\varphi ,\psi \in H^s (\R^N ) $.

We now fix $\varphi \in H^s (\R^N ) $  satisfying~\eqref{fPrelimu}
and we consider $(\varphi _n) _{ n\ge 1 }\subset H^s (\R^N ) $ such that
\begin{equation}  \label{fPrelimc}
 \|\varphi _n- \varphi \| _{ H^s } \goto  _{ n\to \infty  }0.
\end{equation}
We denote by $u$ and $(u_n) _{ n\ge 1 }$ the corresponding solutions of~\eqref{NLS}.

We show that, by letting $\delta _0$ in~\eqref{fPrelimu} possibly smaller (but still independent of $T$ and $\varphi $),
\begin{equation} \label{fLastu}
 \| u_n-u \| _{ L^q((0,T) ,B^s _{ r,2 } )} \goto _{ n\to \infty  }0,
\end{equation}
for all admissible pair $(q,r)$.
Indeed, arguing as in the subcritical case, we see that (see the proof of~\eqref{fConvPpb})
\begin{multline} \label{fConvPpbcrit}
 \|u_n-u \| _{ L^\gamma ((0,T), \dot B^s _{ \rho ,2 } )}
 \le C  \|\varphi _n-\varphi \| _{ \dot H^s } \\+  C
 \|u_n \|  _{ L^{\gamma } ((0,T), L^{\sigma } ) }^\alpha
 \|u_n-u \| _{ L^\gamma ((0,T), \dot B^s _{ \rho ,2 } )}
 +  \|K(u,u_n)\| _{ L^{\gamma '}(0,T) },
\end{multline}
where $K(u,u_n)$ is given by Lemma~\ref{eLemLocLip}.
We next observe that, using~\eqref{fPrelimd}, \eqref{fSupltb}  and possibly choosing $\delta _0$ smaller,
$C  \|u_n \|  _{ L^{\gamma } ((0,T), L^{\sigma } ) }^\alpha  \le 1/2$ for $n$ large.
 We then deduce from~\eqref{fConvPpbcrit}  that
\begin{equation} \label{fConvPpbcritb}
 \|u_n-u \| _{ L^\gamma ((0,T), \dot B^s _{ \rho ,2 } )}
 \le C  \|\varphi _n-\varphi \| _{ \dot H^s }+ C   \|K(u,u_n)\| _{ L^{\gamma '}(0,T) }.
\end{equation}
We claim that
\begin{equation} \label{fSPCru}
 \|K(u, u_n)\| _{ L^{\gamma '}(0,T) }  \goto _{ n\to \infty  }0,
\end{equation}
For proving~\eqref{fSPCru}, we use  the following lemma, whose proof is postponed until the end of this section.

\begin{lem} \label{eKey}
Under the above assumptions, and for $\delta _0$ in~\eqref{fPrelimu} possibly smaller (but still independent of $T$ and $\varphi $) it follows that
\begin{equation} \label{fPrelimcb}
 \|u_n -u\|  _{ L^\gamma ((0,T), L^\sigma ) } \goto  _{ n\to \infty  }0,
\end{equation}
where $\sigma $ is defined by~\eqref{fSuplt}.
\end{lem}

Assuming Lemma~\ref{eKey}, suppose by contradiction that there exist a subsequence, still denoted by $(u_n) _{ n\ge 1 }$ and $\beta  >0$ such that
\begin{equation} \label{fSPCrd}
 \|K(u, u_n)\| _{ L^{\gamma '}(0,T) } \ge \beta
\end{equation}
Note that (see~\eqref{fSPu})
$K (u, u_n)^{\gamma '}\le
C(  \|u_n \| _{ L^\sigma  }^{\frac {\alpha \gamma } {\gamma -2}}+
  \|u \| _{ L^\sigma  }^{\frac {\alpha \gamma } {\gamma -2}}+   \|u\| _{ \dot B^s _{ \rho ,2 }}^{\gamma  } ) $.
Since $\alpha \gamma /(\gamma -2) =\gamma $, this implies
\begin{equation} \label{fSPCrdb}
K (u, u_n)^{\gamma '}\le
C(  \|u_n \| _{ L^\sigma  }^\gamma +
  \|u \| _{ L^\sigma  }^\gamma +   \|u\| _{ \dot B^s _{ \rho ,2 }}^{\gamma  } ) .
\end{equation}
It follows from~\eqref{fPrelimcb} and~\eqref{fSPCrdb} that, after possibly extracting a subsequence, $K (u, u_n)^{\gamma '}$ is bounded by an $L^1$ function. Moreover, we may also assume that $u_n(t)\to u(t)$ in $L^\sigma  (\R^N ) $ for a.a. $t\in (0,T)$, so that by~\eqref{eLemLocLip:q}  $K(u,u_n) \to 0$ for a.a. $t\in (0,T)$. By dominated convergence, we deduce that  $\|K (u, u_n)\| _{ L^{\gamma '}(0,T) } \to 0$. This contradicts~\eqref{fSPCrd}, thus proving~\eqref{fSPCru}.

Applying~\eqref{fConvPpbcritb}, \eqref{fPrelimc} and~\eqref{fSPCru}, we conclude that
\begin{equation*}
 \|u_n-u \| _{ L^\gamma ((0,T), \dot B^s _{ \rho ,2 } )} \goto _{ n \to \infty  }0.
\end{equation*}
By Lemma~\ref{eLemLocLip} (and using again~\eqref{fSPCru}), this implies that
\begin{equation*}
 \|g(u_n) -g(u) \| _{ L^{\gamma '}((0,T), \dot B^s _{ \rho ',2 } )} \goto _{ n \to \infty  }0.
\end{equation*}
 Applying Strichartz estimates, we conclude that~\eqref{fLastu} holds.

To conclude the proof, we argue as follows.
 We let $ \widetilde{T} $ be the supremum of all $0<T<\Tma (\varphi )$ such that if $\varphi _n\to \varphi $ in $H^s (\R^N ) $, then $\Tma (\varphi _n)>T$ for all sufficiently large $n$ and
$u_n \to u$ in $L^q((0,T), B^s _{ r,2 }(\R^N ))$ as $n\to \infty $ for all admissible pair $(q,r)$.
We have just shown that $ \widetilde{T} >0$. We claim that $ \widetilde{T}=\Tma (\varphi ) $. Indeed, otherwise $ \widetilde{T} <\Tma$. Since $u\in C([0, \widetilde{T} ], H^s (\R^N ) )$,
it follows that $\union  _{ 0\le t\le  \widetilde{T}  } \{ u(t) \}$ is a compact subset of $H^s (\R^N ) $.
Therefore, it follows from Strichartz estimates that there exists $T >0$ such that
\begin{equation} \label{fLastd}
\sup  _{ 0\le t\le  \widetilde{T}  }  \| e^{i \cdot \Delta } u(t) \| _{ L^\gamma ((0,T ), B^s _{ \rho ,2 }) }
\le \delta _0.
\end{equation}
We fix $0< \tau < \widetilde{T} $ such that $\tau +T > \widetilde{T} $. If $\varphi _n \to \varphi $ in $H^s (\R^N ) $, it follows from the definition of $ \widetilde{T} $ that $\Tma (\varphi _n) >\tau $ for all sufficiently large $n$ and that $u_n(\tau )\to u(\tau )$ in $H^s (\R^N ) $ as $n\to \infty $.
We then deduce from~\eqref{fLastd} and what precedes that $\Tma (u_n(\tau )) >T$ for all large $n$ and that $u_n (\tau + \cdot )\to u(\tau +\cdot )$ in $L^q((0,T), B^s _{ r,2 }(\R^N ))$
 for all admissible pair $(q,r)$. Thus we see that $\Tma (\varphi _n) >\tau +T$ for all large $n$ and that $u_n  \to u $ in $L^q((0,\tau +T), B^s _{ r,2 }(\R^N ))$  for all admissible pair $(q,r)$.
 This contradicts the definition of $ \widetilde{T} $. Thus $ \widetilde{T} = T$, which proves properties~\eqref{eMain:u} and~\eqref{eMain:d}.

To complete the proof, it thus remains to prove Lemma~\ref{eKey}.

\begin{proof} [Proof of Lemma~$\ref{eKey}$]
Given $2\le r< N/s$, we define $\nu (r) \in (r, \infty )$ by
\begin{equation} \label{fDSig}
\frac {1} {\nu (r)}= \frac {1} {r} - \frac {s} {N},
\end{equation}
so that
\begin{equation} \label{fDSigd}
B^s _{ r,2 }(\R^N )\hookrightarrow L^{\nu  (r)} (\R^N ),
\end{equation}
by Sobolev's embedding.
Note also that if $(\gamma ,\rho )$ is the admissible pair given by~\eqref{fchysecpzt}, then $\nu (\rho )= \sigma $ given by~\eqref{fSuplt}.

We first observe that we need only prove that
\begin{equation} \label{fPrelimcbdu}
 \|u_n -u\|  _{ L^{q_0} ((0,T), L^{\nu (r_0)} ) } \goto  _{ n\to \infty  }0,
\end{equation}
for some admissible pair $(q_0, r_0)$ for which $r_0<\frac {N} {s}$ (and $\nu (r_0)$ is defined by~\eqref{fDSig}).
Indeed, suppose~\eqref{fPrelimcbdu} hold. If $\sigma \ge \nu (r_0)$, then $\rho \ge r_0$.
We fix $ \rho < \widetilde{r} < \frac {N} {s}$ and we deduce
from~\eqref{fDSig} and H\"older's inequality in space and time that
\begin{equation*}
 \| u_n-u \| _{ L^\gamma ((0,T), L^\sigma ) } \le   \| u_n-u \| _{ L^{q_0}((0,T), L^{\nu (r_0)} )}^\theta
  \| u_n-u \| _{ L^{ \widetilde{q} }((0,T), L^{\nu ( \widetilde{r} )} )}^{1- \theta },
\end{equation*}
where $1/\rho = \theta /r_0 + (1-\theta )/  \widetilde{r} $.
Since $  \| u_n-u \| _{ L^{ \widetilde{q} }((0,T), L^{\nu ( \widetilde{r} )} )}^{1- \theta }$ is bounded by~\eqref{fPrelimt} and Sobolev's embedding~\eqref{fDSigd}, we see that~\eqref{fPrelimcb} follows. In the case $\sigma < \nu (r_0)$, then $\rho <r_0$ and we apply the same argument, this time with $( \widetilde{q},  \widetilde{r})= (\infty ,2)  $.

We now prove~\eqref{fPrelimcbdu} for $(q_0, r_0)$ defined by
\begin{equation} \label{fRDu}
q_0= \frac {2\alpha (\alpha +2)} {4-(N-2) \alpha }, \quad r_0= \frac {N(\alpha +2)} {N+s(\alpha +2)}.
\end{equation}
It is straightforward to check that $(q_0, r_0)$ is an admissible pair, that $r_0<\frac {N} {s}$ and that
\begin{equation} \label{fRDd}
\nu (r_0)= \alpha +2.
\end{equation}
We now use a Strichartz-type estimate for the non-admissible pair $(q_0, \alpha +2)$.
More precisely,  it follows from Lemma~2.1 in~\cite{RapDec}\footnote{for more general estimates of this type, see~\cite{Katot, Foschi, Vilela}.}
that there exists a constant
$C$ independent of $n$ and $T$ such that
\begin{equation} \label{fRDq}
\| {\mathcal G} (u_n)- {\mathcal G} (u)\| _{ L^{q_0} (\R, L^{\alpha +2})}\le C
  \| g(u_n)- g(u)\| _{ L^{ \frac {q_0} {\alpha +1} }(\R, L^{\frac {\alpha +2} {\alpha +1}} )}.
\end{equation}
On the other hand, we deduce from~\eqref{fCritub} and H\"older's inequality in space and time that
\begin{multline*}
 \|g(u_n) - g(u)\|_{ L^{ \frac {q_0} {\alpha +1} }(\R, L^{\frac {\alpha +2} {\alpha +1}} )} \\
 \le C (  \| u_n\| _{ L^{q_0} ((0, T), L^{\alpha +2}) }^\alpha  +
\| u\| _{ L^{q_0} ((0, T), L^{\alpha +2}) }^\alpha )  \|u_n-u\| _{ L^{q_0} ((0, T), L^{\alpha +2}) }.
\end{multline*}
Applying~\eqref{fDSigd} and~\eqref{fRDd}, we obtain
\begin{multline}  \label{fRDc}
 \|g(u_n ) - g(u)\|_{ L^{ \frac {q_0} {\alpha +1} }(\R, L^{\frac {\alpha +2} {\alpha +1}} )} \\
 \le C (  \| u_n\| _{ L^{q_0} ((0, T), B^s _{ r_0,2 }) }^\alpha  +
\| u\| _{ L^{q_0} ((0, T), B^s _{ r_0, 2 }) }^\alpha )  \|u_n-u\| _{ L^{q_0} ((0, T), L^{\alpha +2}) }.
\end{multline}
The pair $(q_0,r_0)$ being admissible, we deduce from Strichartz estimate~\eqref{fSBesu},
So\-bo\-lev's embedding~\eqref{fDSigd} and the identity~\eqref{fRDd}  that
\begin{equation}  \label{fRDs}
 \| e^{i \cdot \Delta  } (\varphi _n-\varphi ) \| _{ L^{q_0} (\R, L^{\alpha +2})} \le C  \|\varphi _n -\varphi  \| _{ H^s }.
\end{equation}
Using the equation~\eqref{INLS} for $u_n$ and $u$, together with~\eqref{fRDs}, \eqref{fRDq} and~\eqref{fRDc}, we see that
\begin{multline}  \label{fPrelimhb}
 \|u_n-u\| _{ L^{q_0} ((0, T), L^{\alpha +2}) } \le C  \|\varphi _n- \varphi \| _{ H^s } \\
 +  C  (  \| u_n\| _{ L^{q_0} ((0, T), B^s _{ r_0,2 }) }^\alpha  +
\| u\| _{ L^{q_0} ((0, T), B^s _{ r_0, 2 }) }^\alpha )  \|u_n -u\| _{ L^{q_0} ((0, T), L^{\alpha +2}) }.
\end{multline}
We now observe that if $\varphi $ satisfies~\eqref{fPrelimu}, then $\varphi _n$ also  satisfies~\eqref{fPrelimu} for $n$ large (see~\eqref{fSplzu}), so that by~\eqref{fPrelimt}
we have
\begin{equation} \label{fPrelimtbu}
\max\{  \| u \|_{ L^{q_0} ((0, T), B^s _{ r_0, 2 }) },
\| u_n \|_{ L^{q_0} ((0, T), B^s _{ r_0, 2 }) }\} \le \delta C(q_0, r_0),
\end{equation}
for all sufficiently large $n$.
Therefore, by possibly choosing $\delta _0$ smaller, we can absorb the last term in~\eqref{fPrelimhb} by the left-hand side, and we deduce that~\eqref{fPrelimcbdu} holds. This completes the proof.
\end{proof}

\section{Proof of Corollary~$\ref{eMaind}$} \label{ProofPwr}

The first statement of Corollary~\ref{eMaind} follows from Theorem~\ref{eMain}.
The Lipschitz dependence in the case $\alpha \ge 1$ follows from the estimate~\eqref{feRemBSd}.
Indeed, if $\rho $ is defined by~\eqref{fchysecpzt} then,  applying~\eqref{feRemBSd} with $p= \rho '$, $r=\rho $ and $q=2$, we obtain
\begin{equation*}
 \|g(v)- g(u)\| _{ \dot B^s _{ \rho ',2 } } \le C  \|v\| _{ L^\sigma  }^\alpha  \|v-u \| _{ \dot B^s _{ \rho ,2 } }
 + C  \|u\|_{ \dot B^s _{ \rho ,2 } } ( \|u\| _{ L^\sigma  }^{\alpha -1} +  \|v\| _{ L^\sigma  }^{\alpha -1} )
  \|u-v\| _{ L^\sigma  },
\end{equation*}
where $\sigma $ is given by~\eqref{fSuplt}.
Using the embedding~\eqref{fSupltb}, we deduce that
\begin{equation} \label{fEstLip}
 \|g(v)- g(u)\| _{ \dot B^s _{ \rho ',2 } } \le C
 (\|u \| _{ \dot B^s _{ \rho ,2 } }^\alpha + \|v \| _{ \dot B^s _{ \rho ,2 } }^\alpha )
 \|v-u \| _{ \dot B^s _{ \rho ,2 } }.
\end{equation}
In view of~\eqref{fEstLip}, it is easy to see that one can go through the local existence argument
by using Banach's fixed point theorem with the distance ${\mathrm d}(u,v) =  \|u-v\| _{ L^\gamma ((0,T), B^s _{ \rho ,2 }) }$ instead of  ${\mathrm d}(u,v) =  \|u-v\| _{ L^\gamma ((0,T),  L^\rho ) }$.
See for example the proofs of Theorems~4.9.1 (subcritical case) and~4.9.7 (critical case)
in~\cite{CLN} for details. It now follows from standard arguments that the resulting flow is Lipschitz in the sense of  Corollary~\ref{eMaind}.

\section*{Acknowledgments}
The authors thank W~Sickel for useful references on superposition operators in Besov spaces; and
 the anonymous referee for constructive remarks concerning the exposition of this article.


\begin{thebibliography}{99}

\bibitem{BerghL}{Bergh J. and L\"ofstr\"om J.:} {{\it Interpolation
spaces}, Springer, New York, 1976.}

\bibitem{BourdaudL}{Bourdaud G. and Lanza de Cristoforis M.:} {Regularity of the symbolic calculus in Besov algebras, Studia Math.  {\bf 184}   (2008),  no. 3, 271--298.}

\bibitem{BourdaudLS}{Bourdaud G., Lanza de Cristoforis M. and Sickel W.:} {Superposition operators and functions of bounded $p$-variation, Rev. Mat. Iberoam.  {\bf 22}   (2006),  no. 2, 455--487.}
\bibitem{BourdaudLSd}{Bourdaud G., Lanza de Cristoforis M. and Sickel W.:} {Superposition operators and functions of bounded $p$-variation. II,
Nonlinear Anal. {\bf 62}  (2005), no. 3, 483--517.}

\bibitem{BourdaudMS}{Bourdaud G., Moussai M and Sickel W.:}
{Composition operators on Lizorkin-Triebel spaces, preprint, 2009.}

\bibitem{BrezisM}{Brezis H. and Mironescu P.:}
{Gagliardo-Nirenberg, composition and products in fractional Sobolev spaces.
Dedicated to the memory of Tosio Kato,
J. Evol. Equ. {\bf 1} (2001), no. 4, 387--404. }

\bibitem{CLN}{Cazenave T.:} {{\it Semilinear
Schr\"o\-din\-ger  equations}, Courant Lecture Notes in Mathematics, {\bf 10}.
New York University, Courant Institute of Mathematical Sciences, New York;
American Mathematical Society, Providence, RI, 2003.}

\bibitem{CWHs}{Cazenave T. and Weissler F. B.:} {The Cauchy problem for the
critical nonlinear Schr\"o\-din\-ger equation in $H^{s}$, Nonlinear Anal.  {\bf 14} (1990), no.~10,
807--836.}

\bibitem{RapDec}{Cazenave T. and Weissler F. B.:} {Rapidly decaying solutions of the nonlinear Schr\"o\-din\-ger equation,
Comm. Math. Phys.  {\bf 147} (1992), no.~1, 75--100.}

\bibitem{Foschi}{Foschi D.:} {Inhomogeneous Strichartz estimates,
J. Hyperbolic Differ. Equ. {\bf 2} (2005), no.~1, 1--24.}

\bibitem{FurioliT}{Furioli G. and Terraneo E.:} {Besov spaces and unconditional
well-posedness for the nonlinear Schr\"o\-din\-ger equation, Commun. Contemp.
Math. {\bf 5} (2003), no. 3 349--367.}

\bibitem{Katod}{Kato T.:} {On nonlinear Schr\"o\-din\-ger equations, Ann. Inst. H.~Poin\-ca\-r\'e Phys. Th\'eor. {\bf 46} (1987), no. 1, 113--129.}

\bibitem{Katou}{Kato T.:} {On nonlinear Schr\"o\-din\-ger equations, II.
$H^s$-solutions and unconditional well-posedness,
J. Anal. Math. {\bf 67} (1995), 281--306.}

\bibitem{Katot}{Kato T.:}  {An $L^{q,r}$-theory for nonlinear Schr\"o\-din\-ger
equations, in {\it Spectral and Scattering Theory and Applications}, Advanced
Studies in Pure Mathematics {\bf 23} (1994), 223--238.}

\bibitem{KeelT}{Keel M. and Tao T.:} {Endpoint Strichartz inequalities,
Amer. J. Math.  {\bf 120} (1998), 955--980.}

\bibitem{KenigM}{Kenig C. and Merle F.:} {Global well-posedness, scattering and blow-up
for the energy-critical, focusing, non-linear Schr\"o\-din\-ger  equation in
the radial case,  Invent. Math.  {\bf 166}  (2006),  no.~3, 645--675.}

\bibitem{KillipV}{Killip R. and Visan M.:} {Nonlinear Schr\"o\-din\-ger  equations at critical regularity,  Clay Mathematics Proceedings, vol.  {\bf 10}, 2009.}

\bibitem{Pecher}{Pecher H.:} {Solutions of semilinear
Schr\"o\-din\-ger equations in $H^s$, Ann. Inst. H.~Poin\-ca\-r\'e Phys. Th\'eor. {\bf 67} (1997), no.~3, 259--296.}

\bibitem{Rogers}{Rogers K.M.:}
{Unconditional well-posedness for subcritical NLS in $H^s$,
C. R. Math. Acad. Sci. Paris {\bf 345}  (2007), no. 7, 395--398.}

\bibitem{RunstS}{Runst T and Sickel W.:}  {{\it Sobolev spaces of fractional order, Nemytskij operators, and nonlinear partial differential equations},  de Gruyter Series in Nonlinear Analysis and Applications, 3. Walter de Gruyter \& Co., Berlin, 1996.}

\bibitem{Strichartz}{Strichartz M.:} {Restrictions of Fourier transforms to
quadratic surfaces and decay of solutions of wave equations, Duke Math. J.
{\bf 44} (1977), 705--714.}

\bibitem{TaoV}{Tao T. and Visan M.:} {Stability of energy-critical nonlinear Schr\"o\-din\-ger  equations in high dimensions.  Electron. J. Differential Equations  {\bf 2005}, No. 118, 28 pp. }

\bibitem{Triebel}{Triebel H.:} {{\it Theory of function spaces}, Monographs in
Mathematics  {\bf 78}, Birkh\"auser Verlag, Basel, 1983.}

\bibitem{Tsutsumi}{Tsutsumi Y.:} {$L^{2}$-solutions for nonlinear Schr\"o\-din\-ger
equations and nonlinear groups, Funkcial. Ekvac.  {\bf 30} (1987), no.~1, 115--125.}

\bibitem{Vilela}{Vilela M.-C.:} {Inhomogeneous Strichartz estimates for the
Schr\"o\-din\-ger equation, Trans. Amer. Math. Soc.  {\bf 359} (2007), 2123--2136.}

\bibitem{WinT}{Win Y.Y.S. and Tsutsumi Y.:}
{Unconditional uniqueness of solution for the Cauchy problem of the nonlinear Schršdinger equation,  Hokkaido Math. J. {\bf 37}  (2008), no. 4, 839--859.}

\bibitem{Yajima}{Yajima K.:} {Existence of solutions for Schr\"o\-din\-ger
evolution equations, Comm. Math. Phys.  {\bf 110} (1987), 415--426.}

\end{thebibliography}
\end{document}